\documentclass[12pt]{amsart}
\usepackage{amssymb,latexsym}
\usepackage{enumerate}
\usepackage{dsfont}
\usepackage{a4wide}
\usepackage{mathrsfs}
\usepackage[cp1250]{inputenc}
\usepackage[T1]{fontenc}
\usepackage{tabularx}
\usepackage{multirow}
\usepackage{scalerel}
\usepackage{etoolbox}
\patchcmd{\section}{\normalfont}{\normalfont\Large}{}{}
\patchcmd{\section}{\scshape}{\bfseries}{}{}
\makeatletter
\renewcommand{\@secnumfont}{\bfseries}
\makeatother
\usepackage{lmodern}
\usepackage{hyperref}
\hypersetup{hypertexnames=false}
\let\originalforall=\forall
\renewcommand{\forall}{\mathop{\vcenter{\hbox{\Large$\originalforall$}}}}

\let\originalexists=\exists
\renewcommand{\exists}{\mathop{\vcenter{\hbox{\Large$\originalexists$}}}}

\newtheorem{thm}{Theorem}[]
\newtheorem{cor}[thm]{Corollary}

\date{\today}
\title{The spectral radius formula for Fourier-Stieltjes algebras}
\author{Przemys\l aw Ohrysko}
\address{Chalmers University of Technology and the University of Gothenburg}
\email{p.ohrysko@gmail.com}
\thanks{Supported by foundations managed by The Royal Swedish Academy of Sciences}
\author{Maria Roginskaya}
\address{Chalmers University of Technology and the University of Gothenburg}
\email{maria.roginskaya@chalmers.se}
\begin{document}
\baselineskip=21pt
\begin{abstract}
In this short note we first extend the validity of the spectral radius formula obtained in \cite{ag} to Fourier--Stieltjes algebras. The second part is devoted to showing that for the measure algebra on any locally compact non-discrete Abelian  group there are no non-trivial constraints between three quantities: the norm, the spectral radius and the supremum of the Fourier--Stieltjes transform even if we restrict our attention to measures with all convolution powers singular with respect to Haar measure.
\end{abstract}

\maketitle
\section{Introduction}
We collect first some basic facts from Banach algebra theory and harmonic analysis in order to fix the notation (our main reference for Banach algebra theory is \cite{z}, for harmonic analysis check \cite{r}). For a commutative unital Banach algebra $A$, the Gelfand space of $A$ (the set of all multiplicative-linear functionals endowed with weak$^{\ast}$-topology) will be abbreviated $\triangle(A)$ and the Gelfand transform of an element $x\in A$ is a surjection $\widehat{x}:\triangle(A)\rightarrow\sigma(x)$ defined by the formula: $\widehat{x}(\varphi)=\varphi(x)$ for $\varphi\in\triangle(A)$ where $\sigma(x):=\{\lambda\in\mathbb{C}:\mu-\lambda\mathrm{1}\text{ is not invertible}\}$ is the spectrum of an element $x$.
Let $G$ be a locally compact Abelian group with its unitary dual $\widehat{G}$ and let $M(G)$ denote the Banach algebra of all complex-valued Borel regular measures equipped with the convolution product and the total variation norm. The Fourier-Stieltjes transform will be treated as a restriction of the Gelfand transform to $\widehat{G}$. $M(G)$ is also equipped with involution $\mu\mapsto\widetilde{\mu}$ where $\widetilde{\mu}(E):=\overline{\mu(-E)}$ for every Borel set $E\subset G$. A measure $\mu$ is Hermitian if $\mu=\widetilde{\mu}$ or equivalently, if its Fourier-Stieltjes transform is real-valued. The ideal of measures with Fourier-Stieltjes transforms vanishing at infinity is denoted by $M_{0}(G)$.
\\
Note that we have a direct sum decomposition $M(G)=L^{1}(G)\oplus M_{s}(G)$ where $L^{1}(G)$ is the group algebra identified via Radon-Nikodym theorem with the ideal of all absolutely continuous measures and $M_{s}(G)$ is a subspace consisting of singular (supported on a set of Haar measure zero) measures. For $\mu\in M(G)$ we will write $\mu=\mu_{a}+\mu_{s}$ with $\mu_{a}\in L^{1}(G)$ and $\mu_{s}\in M_{s}(G)$.
\\
The main result of the first section is presented in the framework of Fourier--Stieltjes algebras so we recall some basic information (the standard reference for this part is \cite{ey}, see also recent monograph \cite{kl}).
Let $G$ be a locally compact group and let $B(G)$ be the Fourier--Stieltjes algebra on the group $G$ i.e. the linear span of positive-definite continuous functions on $G$ equipped with the norm given by the duality $B(G)=\left(C^{\ast}(G)\right)^{\ast}$ where $C^{\ast}(G)$ is the full group $C^{\ast}$-algebra of $G$. It is well-known that $B(G)$ with pointwise product is a commutative semisimple unital Banach algebra with $G$ embedded in $\triangle(B(G))$ via point-evaluation functionals. As for the measure algebra we have an orthogonal decomposition of $B(G)=A(G)\oplus B_{s}(G)$ where $A(G)$ is the Fourier algebra of the group $G$ and $B_{s}(G)$ is a subspace consisting of singular elements (it was proved first in \cite{ar}, see \cite{ow} for modern presentation). For $f\in B(G)$ we will write $f=f_{a}+f_{s}$ with $f_{a}\in A(G)$ and $f_{s}\in B_{s}(G)$.
\\
In \cite{ag} the following spectral radius formula for measure algebra on compact (not necessarily Abelian) group was shown to be true.
\begin{thm}\label{wzo}
Let $G$ be a compact group and let $\mu\in M(G)$. Then the following formula holds:
\begin{equation*}
r(\mu)=\max\{\sup_{\sigma\in\widehat{G}}r(\widehat{\mu}(\sigma)),\inf_{n\in\mathbb{N}}\|\left(\mu^{\ast n}\right)_{s}\|^{\frac{1}{n}}\}.
\end{equation*}
\end{thm}
Here $\widehat{G}$ is the set of all equivalence classes of irreducible unitary representations of the group $G$ and $\widehat{\mu}(\sigma)$ is the matrix-valued Fourier--Stieltjes transform.
\\
In the first section we will give a short proof of the counterpart of this formula for Fourier-Stieltjes algebras. The aim of the second part is to construct examples of measures exhibiting that there are no other connections between the norm, the spectral radius and the supremum of the Fourier--Stieltjes transform than the obvious ones.
\section{The spectral radius formula}
\begin{thm}
Let $G$ be a locally compact group and let $f\in B(G)$. Then the following formula holds true:
\begin{equation}\label{wz}
r(f)=\max\{\|f\|_{\infty},\inf_{n\in\mathbb{N}}\|\left(f^{n}\right)_{s}\|^{\frac{1}{n}}\}
\end{equation}
\end{thm}
\begin{proof}
Clearly, $r(f)\geq \|f\|_{\infty}$. Let us take $n\in\mathbb{N}$. Then $f^{n}=\left(f^{n}\right)_{a}+(f^{n})_{s}$. This implies $\|f^{n}\|=\|\left(f^{n}\right)_{a}\|+\|\left(f^{n}\right)_{s}\|\geq\|\left(f^{n}\right)_{s}\|$.
Taking the infimum on both sides and applying the spectral radius formula we obtain the first inequality.
\\
As $A(G)$ is an ideal in $B(G)$ we have the splitting of $\triangle(B(G))$:
\begin{equation}\label{roz}
\triangle(B(G))=\triangle(A(G))\cup h(A(G))=G\cup h(A(G)),
\end{equation}
where $h(A(G))=\{\varphi\in\triangle(B(G)):\varphi|_{A(G)}=0\}$.

For $f\in B(G)$ and any $\varphi\in\triangle(B(G))$ we have $\varphi(f)=\varphi(f_{a})+\varphi(f_{s})$. If $\varphi\in\triangle(A(G))$ then, of course, $|\varphi(f)|=|f(x)|$ for some $x\in G$ and so $|\varphi(f)|\leq \|f\|_{\infty}$. In case $\varphi\in h(A(G))$ we have $|\varphi(f)|=|\varphi(f_{s})|\leq r(f_{s})\leq \|f_{s}\|$. The above considerations are summarized as the following inequality:
\begin{equation*}
r(f)\leq\max\{\|f\|_{\infty},\|f_{s}\|\}
\end{equation*}
Let us fix $n\in\mathbb{N}$ and apply this inequality to $f^{n}$. As $r(f^{n})=r(f)^{n}$ we get
\begin{equation*}
r(f)^{n}\leq\max\{\|f\|_{\infty}^{n},\|\left(f^{n}\right)_{s}\|\}.
\end{equation*}
To finish the proof of (\ref{wz}) we need only to take the $\frac{1}{n}$ power on both sides and infimum over $n$.
\end{proof}
\begin{cor}\label{dwa}
Let $f\in B(G)$ satisfy $r(f)>\|f\|_{\infty}$. Then at least one of the following situations occur:
\begin{enumerate}
  \item For every $n\in\mathbb{N}$ we have $f^{n}\in B_{s}(G)$.
  \item $\|f\|>r(f)$.
\end{enumerate}
\end{cor}
\begin{proof}
Suppose that $f^{n_{0}}\notin B_{s}(G)$ for some $n_{0}\in \mathbb{N}$. Then by Theorem \ref{wzo} we obtain
\begin{equation*}
r(f)=\inf_{n\in\mathbb{N}}\|(f^{n})_{s}\|^{\frac{1}{n}}\leq\|(f^{n_{0}})_{s}\|^{\frac{1}{n_{0}}}<\|f^{n_{0}}\|^{\frac{1}{n_{0}}}\leq\|f\|,
\end{equation*}
which finishes the proof.
\end{proof}
Our last aim for this section is to show that neither of the two items in Corollary \ref{dwa} implies the other and also that both of them may hold at the same time providing explicit examples of measures on the circle group. The last one will be also used in the next section.
\begin{itemize}
  \item Example of a measure satisfying: $\|\mu\|>r(\mu)>\|\widehat{\mu}\|_{\infty}$ but with non-singular convolution powers:
  \\
  Let us consider $\mu=\frac{1}{2}R-\frac{1}{2}\mathrm{m}$, where $R=\prod_{k=1}^{\infty}(1+\cos(3^{k}t))$ is the classical Riesz product measure (understood as a weak$^{\ast}$ limit of finite products) and $m$ is the normalized Lebesgue measure on $\mathbb{T}$. Let us recall that $R$ is a continuous probability measure with independent powers (i.e. $R^{\ast n}\bot R^{\ast m}$ for $n\neq m$, see \cite{bm}) and satisfying $\sigma(R)=\overline{\mathbb{D}}=\{z\in\mathbb{C}:|z|\leq 1\}$ (check the chapter of \cite{grmc} on Riesz products). Moreover,
  \begin{equation*}
  \widehat{R}(\mathbb{Z})=\{0\}\cup\left\{\frac{1}{2^{n}}\right\}_{n\in\mathbb{N}}\cup \{1\}.
  \end{equation*}
  We clearly have $\|\mu\|=1$. It is also immediate that $\|\widehat{\mu}\|_{\infty}=\frac{1}{4}$. Taking into account (\ref{roz}) we get
  \begin{gather*}
  \sigma(\mu)=\widehat{\mu}(\mathbb{Z})\cup\widehat{\mu}(\triangle(M(\mathbb{T}))\setminus\mathbb{Z})=\\
  =\{0\}\cup\left\{\frac{1}{2^{n+1}}\right\}_{n=1}^{\infty}\cup\frac{1}{2}\widehat{R}(\triangle(M(\mathbb{T}))\setminus\mathbb{Z})=\\
  =\frac{1}{2}\overline{\mathbb{D}}\text{ as }\sigma(R)=\overline{\mathbb{D}}.
  \end{gather*}
  \item Example of a measure with all convolution powers singular with the properties $\|\mu\|=r(\mu)>\|\widehat{\mu}\|_{\infty}$.
  \\
  Put $\mu=\frac{1}{2}\left(R-R^{\ast 2}\right)$. Then $\|\mu\|=1$ and all convolution powers of $\mu$ are singular. By the spectral properties of Riesz products there exists $\varphi\in\triangle(M(\mathbb{T}))$ such that $\varphi(R)=-1$. This gives $\varphi(\mu)=-1$ and thus $r(\mu)=1$. Also, $\|\widehat{\mu}\|_{\infty}=\frac{1}{8}$.
  \item Example of a measure with all convolution powers singular satisfying additionally $\|\mu\|>r(\mu)>\|\widehat{\mu}\|_{\infty}$.
  \\
  Let $q_{0}(z):=\frac{1}{4}\left(z^{5}-z^{4}+z^{2}-z\right)=\frac{1}{4}z(z^{3}+1)(z-1)$. Then $q_{0}(1)=0$ and by the maximum modulus principle we get
  \begin{equation}\label{pierw}
  \sup_{z\in\left\{\frac{1}{2^{n}}\right\}_{n\in\mathbb{N}}\cup\{1\}}|q_{0}(z)|=\sup_{z\in\left\{\frac{1}{2^{n}}\right\}_{n\in\mathbb{N}}}|q_{0}(z)|<\sup_{z\in\overline{\mathbb{D}}}|q_{0}(z)|.
  \end{equation}
  Also, using maximum modulus principle again,
  \begin{equation*}
  \sup_{z\in\overline{\mathbb{D}}}|q_{0}(z)|=\sup_{z\in\mathbb{T}}|q_{0}(z)|\leq\frac{1}{4}\sup_{z\in\mathbb{T}}|z^{3}+1|\cdot \sup_{z\in\mathbb{T}}|z-1|=1.
  \end{equation*}
  However, as the supremum of $|z^{3}+1|$ on $\mathbb{T}$ is attained only for the roots of unity of order three and the supremum of $|z-1|$ as attained only for $z=-1$ we obtain
  \begin{equation}\label{dru}
  \sup_{z\in\overline{\mathbb{D}}}|q_{0}(z)|<1.
  \end{equation}
  It is convenient to use the following convention: for an algebraic polynomial $f(z)=a_{n}z^{n}+a_{n-1}z+\ldots+a_{1}z+a_{0}$ we put
  \begin{equation*}
  |f|_{1}:=\sum_{k=0}^{n}|a_{k}|
  \end{equation*}
  Let $\mu:=q_{0}(R)$ where $R$ is the classical Riesz product. Then $\|\mu\|=|q_{0}|_{1}=1$ as $R$ has independent powers. By the spectral mapping theorem
  \begin{equation*}
  r(\mu)=\sup_{z\in\overline{\mathbb{D}}}|q_{0}(z)|\text{ as }\sigma(R)=\overline{\mathbb{D}}.
  \end{equation*}
  By (\ref{dru}) we thus have $r(\mu)<1$. Using the properties of the functional calculus again we get
  \begin{equation*}
  \|\widehat{\mu}\|_{\infty}=\sup_{z\in\left\{\frac{1}{2^{n}}\right\}_{n\in\mathbb{N}}\cup\{1\}}|q_{0}(z)|.
  \end{equation*}
  The application of (\ref{pierw}) finishes the argument.
\end{itemize}
\section{The construction}
In this section we are going to show that there are no non-trivial constraints between the quantities $\|\mu\|$, $r(\mu)$ and $\|\widehat{\mu}\|_{\infty}$ even if we restrict our attention to measures with all convolution powers singular.
\begin{thm}\label{prz}
Let $a,b$ be two fixed numbers satisfying $0<b\leq a\leq 1$ and let $G$ be a locally compact Abelian non-discrete group. Then there exists a measure $\mu\in M_{0}(G)$ with all convolution powers singular such that $\|\widehat{\mu}\|_{\infty}=b$, $r(\mu)=a$ and $\|\mu\|=1$.
\end{thm}
The proof of this theorem depends on the existence of measures with special properties and we combine Theorem 6.1.1 and Corollary 7.3.2 from \cite{grmc} to obtain the formulation adequate for our needs.
\begin{thm}\label{og}
Let $G$ be a locally compact Abelian non-discrete group. Then there exists Hermitian independent power probability measure in $M_{0}(G)$. Moreover, the spectrum of this measure is the whole unit disc.
\end{thm}
An application of the functional calculus to a measure described in Theorem \ref{og} shows that in order to prove Theorem \ref{prz} it is enough to construct a polynomial with the properties specified below.
\begin{thm}
Let $a,b$ be two fixed numbers satisfying $0<b\leq a\leq 1$. Then there exists an algebraic polynomial $p$ with the following properties:
\begin{enumerate}
  \item $|p|_{1}=1$,
  \item $\sup_{z\in\overline{\mathbb{D}}}|p(z)|=a$,
  \item $\sup_{x\in [-1,1]}|p(x)|=b$, $p(1)=b$
\end{enumerate}
\end{thm}
\begin{proof}
We divide the argument into three steps.
\\
\underline{Step 1}: $\frac{-5+4\sqrt{2}}{7}=:b_{0}<b<a\leq 1$.
\\
The polynomial $q_{0}(z)=\frac{1}{4}z(z-1)(z+1)(z^{2}-z+1)=\frac{1}{4}z(z-1)(z^{3}+1)=\frac{1}{4}\left(z^{5}-z^{4}+z^{2}-z\right)$ introduced in the previous section has the following properties
\begin{itemize}
  \item $|q_{0}|_{1}=1$,
  \item $q_{0}(1)=q_{0}(0)=q_{0}(-1)=0$,
  \item $\|q_{0}\|_{C(\overline{\mathbb{D}})}<1$.
\end{itemize}
Let us consider an inductive procedure:
$q_{n+1}(z)=q_{n}(z)\cdot q_{0}(z^{\mathrm{deg}q_{n}+1})$ with $q_{0}$ defined as before. It is elementary to verify that for every $n\in\mathbb{N}$ the polynomial $q_{n}$ satisfies:
\begin{itemize}
\item $|q_{n}|_{1}=1$,
\item $q_{n}(1)=q_{n}(0)=q_{0}(-1)=0$,
\item $\|q_{n}\|_{C(\overline{\mathbb{D}})}\leq \|q_{0}\|_{C(\overline{\mathbb{D}})}^{n+1}$.
\end{itemize}
The first property follows from the fact that at each step we multiply $q_{n}$ by a polynomial with gaps between powers longer than $\mathrm{deg}q_{n}$.
Moreover, we get an elementary estimate of $q_{n}$ on the real axis $\left(\text{observe first } |q_{n}(x)|\leq |q_{1}(x)|\cdot \|q_{0}\|^{n-1}_{C(\overline{\mathbb{D}})}\right)$:
\begin{equation}\label{nap}
|q_{n}(x)|\leq c(1-x)^{2}(1+x)^{2}\cdot \|q_{0}\|_{C(\overline{\mathbb{D}})}^{n-1}\text{ for $x\in[-1,1]$ and $n\geq 1$},
\end{equation}
where $c$ is a numerical constant.
\\
Consider a family of polynomials $p_{\alpha}$ indexed by the parameter $\alpha\in [b,1]$:
\begin{equation}\label{defg}
p_{\alpha}(z)=\frac{1}{2}(\alpha+b)z^{4}-\frac{1}{2}(\alpha-b)z^{2}+(1-\alpha)q_{n}(z),
\end{equation}
where $n>1$ will be chosen later but let's take into account that the form of $q_{n}$ implies $|p_{\alpha}|_{1}=1$ for $\alpha\in [b,1]$. Clearly, we also get $p_{\alpha}(1)=p_{\alpha}(-1)=b$. We apply the
elementary calculus to the function $f(x)=\frac{1}{2}(\alpha+b)x^{4}-\frac{1}{2}(\alpha-b)x^{2}$ on the interval $[0,1]$ (it is enough as $f$ is even and the estimate of $|p_{\alpha}|$ is also even-type) finding local extremum at the point $x_{0}=\sqrt{\frac{\alpha-b}{2(\alpha+b)}}$ equal to
\begin{equation*}
f(x_{0})=-\frac{1}{8}\frac{(\alpha-b)^{2}}{\alpha+b}.
\end{equation*}
Taking the supremum over $\alpha\in [b,1]$ we obtain $|f(x_{0})|\leq \frac{1}{8}\frac{(1-b)^{2}}{1+b}$. Solving the inequality $\frac{1}{8}\frac{(1-b)^{2}}{1+b}<b$ gives $b>\frac{-5+4\sqrt{2}}{7}=b_{0}$ implying
$|f(x_{0})|<b$ for $b>b_{0}$.
\\
Let us fix $\varepsilon$ satisfying
\begin{equation}\label{eps}
\varepsilon<\min\left\{a-b,b-\frac{1}{8}\frac{(1-b)^{2}}{1+b}\right\}
\end{equation}
and take $n$ (depending only on $\varepsilon$) such that $\|q_{n}\|_{C(\overline{\mathbb{D}})}<\varepsilon$.
As $|p_{\alpha}(x_{0})|\leq |f(x_{0})|+\|q_{n}\|_{C(\overline{\mathbb{D}})}$, we obtain $|p_{\alpha}(x_{0})|<b$ for $\alpha\in [b,1]$
\\
Further analysis of $f$ shows that the function $x\mapsto |f(x)|$ has the following properties:
\begin{itemize}
  \item $f(0)=0$,
  \item Increases on the interval $[0,x_{0}]$ up to value $|f(x_{0})|$.
  \item Decreases on the interval $[x_{0},x_{1}]$ up to value $0$, where $x_{1}=\sqrt{\frac{\alpha-b}{\alpha+b}}$.
  \item Increases on the interval $[x_{1},1]$ up to value $b$ and $f(x)>0$ in this region.
\end{itemize}
It follows from the above discussion that for $x\in [0,x_{1}]$ we are allowed to use the same estimates as for the point $x_{0}$ to obtain the conclusion $|p_{\alpha}(x)|\leq b$. For the interval $[x_{1},1]$ we proceed as follows (here we use (\ref{nap})):
\begin{gather*}
|p_{\alpha}(x)|\leq |f(x)|+(1-\alpha)c\|q_{0}\|_{C(\overline{\mathbb{D}})}^{n-1}(1-x)^{2}(1+x)^{2} = \\
= \frac{1}{2}x^{2}\left((\alpha+b)x^{2}-(\alpha-b)\right)+(1-\alpha)c\|q_{0}\|_{C(\overline{\mathbb{D}})}^{n-1}(1-x)^{2}(1+x)^{2}\leq\\
\leq bx^{2}+\widetilde{c}\|q_{0}\|_{C(\overline{\mathbb{D}})}^{n-1}(1-x)^{2},
\end{gather*}
where $\widetilde{c}$ is another numerical constant. The last expression can be made smaller than $b$ by taking sufficiently big $n$ (depending on $b$ only) as it is a non-negative quadratic function with positive leading coefficient so one has to take care of the endpoints of the interval $[0,1]$ for which we have values $\widetilde{c}\|q_{0}\|_{C(\overline{\mathbb{D}})}^{n-1}$ and $b$.
\\
We start the estimates for the uniform norm of $p_{\alpha}$ on the unit disc. The upper bound:
\begin{equation}\label{gor}
\|p_{\alpha}\|_{C(\overline{\mathbb{D}})}\leq \alpha+\|q_{n}\|_{C(\overline{\mathbb{D}})}<\alpha+\varepsilon.
\end{equation}
In order to get the lower bound we simply observe that
\begin{equation}\label{dolnefin}
\|p_{\alpha}\|_{C(\overline{\mathbb{D}})}\geq |p_{\alpha}(i)|=|\alpha+(1-\alpha)q_{n}(i)|.
\end{equation}
Consider now the function $F(\alpha):=\|p_{\alpha}\|_{C(\overline{\mathbb{D}})}$ for $\alpha\in[b,1]$. Using $(\ref{gor})$ and $(\ref{dolnefin})$ we get
\begin{equation*}
F(b)\leq b+\varepsilon<a\text{ and }F(1)=1.
\end{equation*}
It is immediate that $F$ is continuous (in fact Lipschitz continuous) so there exists $\alpha_{0}\in [b,1]$ such that $F(\alpha_{0})=a$ and then the polynomial $p_{\alpha_{0}}$ satisfy the assertion of the theorem.
\\
\underline{Step 2}: $0<b<a\leq 1$.
\\
We find first $k\in\mathbb{N}$ (depending on $b$ only) such that $\sqrt[k]{b}>b_{0}$. Let us fix $\varepsilon$ restricted as in (\ref{eps}) with $b,a$ replaced by $\sqrt[k]{b},\sqrt[k]{a}$ respectively. Construct the polynomial $q_{n}$ starting with the data $\sqrt[k]{b}, \sqrt[k]{a}$ and $\varepsilon$ as in Step 1 and let us form the family $p_{\alpha}$ for $\alpha\in [\sqrt[k]{b},1]$ as in (\ref{defg}). Now we perform another inductive procedure to be executed $k$-times:
\begin{equation*}
w_{l+1,\alpha}(z)=w_{l,\alpha}(z)\cdot w_{0,\alpha}(z^{4\mathrm{deg}w_{l,\alpha}+1})\text{ with }w_{0,\alpha}=p_{\alpha}.
\end{equation*}
As $4\mathrm{deg}w_{l,\alpha}+1>\mathrm{deg}w_{l,\alpha}$ we get $|w_{l,\alpha}|_{1}=1$ for $l\in\mathbb{N}$ and $\alpha\in [\sqrt[k]{b},1]$.
\\
Let us define $w_{\alpha}:=w_{k-1,\alpha}$ for $\alpha\in[\sqrt[k]{b},1]$. Then $w_{\alpha}(1)=\left(p_{\alpha}(1)\right)^{k}=b$ and also $\|w_{\alpha}\|_{C[-1,1]}\leq \|p_{\alpha}\|^{k}_{C([-1,1])}\leq b$. Moreover, by (\ref{gor}), we obtain
\begin{equation*}
\|w_{\alpha}\|_{C(\overline{\mathbb{D}})}\leq (\alpha+\varepsilon)^{k}\leq \alpha^{k}+c\varepsilon,\text{ where $c$ is a numerical constant}
\end{equation*}
Decreasing $\varepsilon$ if necessary we get
\begin{equation}\label{gorfin}
\|w_{\sqrt[k]{b}}\|_{C(\overline{\mathbb{D}})}\leq b+c\varepsilon<a.
\end{equation}
For the lower estimate we use (\ref{dolnefin})
\begin{equation}\label{dolnepo}
\|w_{1}\|_{C(\overline{\mathbb{D}})}\geq |w_{1}(i)|=\prod_{l=0}^{k-1}|p_{1}(i^{4\mathrm{deg}w_{l,1}+1})|=(p_{1}(i))^{k}=1.
\end{equation}
The same continuity argument as in Step 1 using (\ref{gorfin}) and (\ref{dolnepo}) finishes the proof.
\\
\underline{Step 3}: $b=a=1$ and $b=a<1$.
\\
If $b=a=1$ then the polynomial $p(z)=z$ does the job. In case $b=a<1$ we take the polynomial $p$ with the properties $|p|_{1}=1$, $\|p\|_{C(\overline{\mathbb{D}})}=a$ and $\|p\|_{C([-1,1])}=\widetilde{b}$ for some positive $\widetilde{b}<a$ constructed in Steps 1 and 2. Let $z_{max}$ be the point on the circle for which $|p(z_{max})|=\|p\|_{C(\overline{\mathbb{D}})}$ and consider the polynomial $p_{max}(z):=p(z_{max}\cdot z)$. Then $|p_{max}|_{1}=1$, $\|p_{max}\|_{C(\overline{\mathbb{D}})}=\|p\|_{C(\overline{\mathbb{D}})}=a$ and $\|p_{max}\|_{C([-1,1])}=|p_{max}(1)|=|p(z_{max})|=a$.
\end{proof}

\end{document}